\documentclass[12pt]{amsart}

\usepackage{amsmath,amssymb,amsthm,graphicx,hyperref,pinlabel}
\usepackage{enumitem}
\usepackage{tikz}
\usetikzlibrary{arrows}

\usepackage[small,nohug,heads=vee]{diagrams}
\diagramstyle[labelstyle=\scriptstyle]

\usepackage[margin=1.25in]{geometry}

\numberwithin{equation}{section}
\usepackage{hyperref}
\hypersetup{bookmarksdepth=3}

\theoremstyle{plain}
\newtheorem{theorem}[equation]{Theorem}

\newtheorem{lemma}[equation]{Lemma}

\newtheorem*{claim}{Claim}
\newtheorem*{claim*}{Claim}

\newtheorem*{addendum-Thm15}{Addendum to Theorem~\ref{Thm_comparing_RF_weak_form}}

\theoremstyle{remark}

\theoremstyle{definition}

\newcommand{\acts}{\curvearrowright}

\newcommand{\wh}{\widehat}
\newcommand{\Ga}{\Gamma}

\newcommand{\gl}{\operatorname{GL}}

\newcommand{\D}{\partial}

\newcommand{\M}{{\mathcal M}}

\newcommand{\R}{\mathbb R}

\renewcommand{\t}{\mathfrak{t}}

\newcommand{\Z}{\mathbb Z}

\newcommand{\aff}{\operatorname{Aff}}

\newcommand{\ben}{\begin{enumerate}}

\newcommand{\de}{\delta}
\newcommand{\diam}{\operatorname{diam}}
\newcommand{\Diff}{\operatorname{Diff}}

\newcommand{\een}{\end{enumerate}}

\newcommand{\eps}{\epsilon}

\renewcommand{\fill}{\operatorname{fill}}

\newcommand{\id}{\operatorname{id}}

\newcommand{\isom}{\operatorname{Isom}}
\newcommand{\la}{\lambda}

\newcommand{\loc}{\operatorname{loc}}

\newcommand{\met}{\operatorname{Met}}

\newcommand{\nil}{\operatorname{Nil}}

\newcommand{\ol}{\overline}
\newcommand{\om}{\omega}
\newcommand{\Om}{\Omega}

\newcommand{\ra}{\rightarrow}

\newcommand{\Rm}{\operatorname{Rm}}

\renewcommand{\th}{\theta}

\DeclareMathOperator{\Int}{Int}
\DeclareMathOperator{\can}{can}

\DeclareMathOperator{\vol}{vol}

\DeclareMathOperator{\aut}{Aut}

\newcommand{\ov}[1]{\overline{#1}}
\newcommand{\td}[1]{\widetilde{#1}}

\def\XXint#1#2#3{{\setbox0=\hbox{$#1{#2#3}{\int}$}
     \vcenter{\hbox{$#2#3$}}\kern-.5\wd0}}

\begin{document}

\author{Richard H. Bamler}
\address{Department of Mathematics, University of California, Berkeley, Berkeley, CA 94720}
\email{rbamler@berkeley.edu}
\author{Bruce Kleiner}
\address{Courant Institute of Mathematical Sciences, New York University,  251 Mercer St., New York, NY 10012}
\email{bkleiner@cims.nyu.edu}
\thanks{The first author was supported by NSF grant DMS-1906500. \\
\hspace*{2.7mm} The second author was supported by NSF grants DMS-1711556, DMS-2005553, and a Simons Collaboration grant.}

\title{Diffeomorphism groups of prime 3-manifolds}

\date{\today}

\begin{abstract}
Let $X$ be a compact orientable non-Haken  $3$-manifold modeled on the Thurston geometry $\nil$.  We show that  the diffeomorphism group $\Diff(X)$ deformation retracts to the isometry group $\isom(X)$.   Combining this with earlier work by many authors, this  completes the determination the homotopy type of $\Diff(X)$ for any compact, orientable, prime $3$-manifold $X$. 
\end{abstract}

\maketitle

\tableofcontents

\section{Introduction}
For a  compact, connected, smooth $3$-manifold $X$, we denote by $\Diff(X)$
  the group of smooth diffeomorphisms of $X$ equipped with the $C^\infty$-topology.

Our goal in this paper is to prove the following:
\begin{theorem}
\label{thm_main}
Let $X$ be a compact, connected, geometrizable $3$-manifold modeled on the Thurston geometry $\nil$, and $g_0$ be a $\nil$-metric on $X$.  Then the inclusion $\isom(X,g_0)\hookrightarrow \Diff(X)$ is a homotopy equivalence.
\end{theorem}
\noindent
Combined with results from our previous two papers \cite{bamler_kleiner_gsc,bamler_kleiner_psc} and work of many other mathematicians \cite{hatcher_haken,ivanov_haken,ivanov_1,ivanov_2,gabai_smale_conjecture_hyperbolic,rubinstein_et_al,mccullough_soma}, Theorem~\ref{thm_main} completes the project of understanding the topology of $\Diff(X)$, when $X$ is a closed, orientable, prime $3$-manifold; see \cite{bamler_kleiner_gsc,rubinstein_et_al} for references and history.

We remark that in the setting of Theorem~\ref{thm_main} the Lie group $\isom(X,g_0)$ is an extension of a finite group $F$ by a circle $S^1$; it can be characterized using the Seifert structure on $X$.

The proof of Theorem~\ref{thm_main} is based on Ricci flow and follows \cite{bamler_kleiner_gsc} closely, except for some input specific to $\nil$ geometry.  We refer the reader to the introduction of \cite{bamler_kleiner_gsc} for an outline of the proof. 

Together with previous work, we have shown that singular Ricci flow offers a uniform approach to studying diffeomorphism groups of a large class of prime manifolds: spherical space forms, hyperbolic manifolds, $S^2 \times S^1$ and non-Haken manifolds modeled on $\nil$.
We expect that the methods from this paper can be readily adapted to also cover the non-Haken case modeled on $\mathbb{H}^2 \times \R$ and $\td{SL(2,\R)}$ and the Haken case modeled on Solv, which were covered by McCullough-Soma \cite{mccullough_soma} and Hatcher \cite{hatcher_haken}, respectively.
We believe that our methods can be even carried further to cover all geometrizable cases; this would be more technical, however, due to the fact that the asymptotic behavior of the Ricci flow is not fully understood in this case (see \cite[Question~1.7]{Bamler_longtime_0}).

\section{Preliminaries}
\label{sec_prelim}

In this section we cover some of the background material needed for the proof Theorem~\ref{thm_main}; since the proof follows \cite{bamler_kleiner_gsc} closely, the reader may wish to consult the preliminaries section of \cite{bamler_kleiner_gsc} as well.

\subsection{Nil geometry and Nil-structures} \label{subsec_prelim}

We collect some facts about $\nil$ geometry and $\nil$ structures \cite{scott_geometries}.

 Thurston's $\nil$-geometry is the Heisenberg group 
\begin{equation} \label{eq_nil_def}
 \nil := 
 \left\{ \left(\begin{array}{ccc}1 & x_1 & x_3 \\0 & 1 & x_2 \\0 & 0 & 1\end{array}\right) \;\; : \;\; x_1,x_2,x_3 \in \R \right\} ,
\end{equation}
equipped with the left invariant Riemannian metric $g_{\nil}$ for which the basis $X_1:=E_{1,2}$, $X_2:=E_{2,3}$, $X_3:=E_{1,3}$ of elementary matrices is orthonormal; here we use the standard  identification of the Lie algebra $L(\nil)$ with the space of strictly upper triangular matrices.    We have $[X_1,X_2]=X_3$, $[X_1,X_3]=[X_2,X_3]=0$. In what follows we will use $\nil$ to denote the Lie group or the Riemannian manifold  $(\nil,g_{\nil})$, depending on the context. 

 The automorphism group of $L(\nil)$ is generated by elements whose matrix with respect to the basis $X_1,X_2,X_3$ has one of the two forms:
\begin{align}
\label{eqn_gl2_aut}
\left(\begin{array}{ccc}A & 0 \\0 &\det A\end{array}\right) \\
 \left(\begin{array}{ccc}1& 0& 0 \\0 &1&0\\
b_1&b_2&1\end{array}\right)
\end{align}
for some $A\in \gl(2,\R)$, $b_1,b_2\in \R$.  Automorphisms as in \eqref{eqn_gl2_aut} with $A=\la\id$ are (Carnot) dilations.
The affine group  of $\nil$ is the group $\aff(\nil)$ generated by the left translations and the automorphisms; this is the semidirect product of the translations with the automorphism group $\aut(\nil)$.  Every $\Phi\in \aff(\nil)$ carries left invariant vector fields to left invariant vector fields, and we denote by $D\Phi\in \aut(L(\nil))$ the associated Lie algebra automorphism.   The isometry group $\isom(\nil)$ is the set of $\Phi\in  \aff(\nil)$ such that $D\Phi$ is as in \eqref{eqn_gl2_aut} for some $A\in O(2)$.  Carnot dilations normalize $\isom(\nil)$.    Any  left-invariant metric $g$ on $\nil$ is homothetic to $g_{\nil}$; more specifically, there exists  an automorphism $\phi \in \aut (\nil)$ and some $\lambda > 0$ such that $g = \lambda^2 \phi^* g_{\nil}$.

 We identify the abelianization $\nil/\exp(\R X_3)$ with $\R^2$.  Any automorphism of $\nil$ preserves the center $\exp(\R X_3)$, and hence any   $\Phi\in\aff(\nil)$  induces an affine map $\R^2\ra \R^2$.  This yields an $\aff(\nil)$-equivariant fibration $\nil\ra\R^2$; this is also an $\isom(\nil)$-equivariant fibration, when we restrict to $\isom(\nil)\subset \aff(\nil)$.

 A {\bf Nil-structure} on a manifold $X$ is a Riemannian metric $g$ locally isometric to $\nil$; we say that $X$ is {\bf modeled on $\nil$} if it admits a $\nil$-structure.   When $X$ is compact and connected, then the universal cover $(\td X,\td g)$ is isometric to $\nil$, and up to isometry we have $X=\nil/\Ga$ for some lattice $\Ga\subset \isom(\nil)$ acting freely on $\nil$.  Since $\aff(\nil)\acts \nil$ is orientation preserving, $\nil/\Ga$ is orientable.

 Any lattice $\Ga\subset \isom(\nil)$ has a finite index translation subgroup \cite{auslander_bieberbach,raghunathan_discrete_subgroups_lie_groups}.  Any isomorphism $\Ga_1\ra \Ga_2$ between two lattices $\Ga_1,\Ga_2\subset\isom(\nil)$ is induced by (conjugation by) some $\Phi\in \aff(\nil)$  \cite{auslander_bieberbach}.  

If $\Ga\subset\isom(\nil)$ is a lattice acting freely on $\nil$, then as above we obtain an isometric action $\Ga\acts\R^2$, and a quotient Euclidean orbifold $\R^2/\ov \Ga$, where $\ov\Ga$ is the image of $\Ga$ in $\isom(\R^2)$.   The abelianization fibration $\nil\ra\R^2$ is $\Ga$-equivariant, and descends to a Seifert fibration $\nil/\Ga\ra \R^2/\ov\Ga$ over a flat orbifold whose  singular points are only of cone type (i.e., there are no reflectors or corner reflectors).    By conjugating $\Ga$ by a Carnot dilation we may arrange that $\nil/\Ga$ has unit volume.  

An orientable 3-manifold $X$ is called {\bf Haken} if it is prime and contains an embedded, $2$-sided, incompressible surface  (i.e.,  a surface $Y\subset X$ of nonpositive Euler characteristic such that $\pi_1(Y) \to \pi_1(X)$ is injective).

\begin{lemma}
\label{lem_unique_unit_volume_nil}
Suppose $X$ is a compact connected geometrizable non-Haken $3$-manifold modeled on $\nil$.  Then it has a unique unit volume $\nil$-structure, up to isometry.
\end{lemma}
\begin{proof}
Let $\Ga:=\pi_1(X)$, and consider two isometric covering group actions $\Ga\stackrel{\rho,\,\rho'}{\acts}\nil$ with quotients diffeomorphic to $X$ and of unit volume; we denote by $\Ga\stackrel{\ov\rho\,,\,\ov\rho'}{\acts}\R^2$ the induced isometric actions on $\R^2$, and $\Ga\stackrel{D\rho,\,D\rho'}{\acts}L(\nil)$, $\Ga\stackrel{D\ov\rho,\,D\ov\rho'}{\acts}\R^2$ the associated linear actions obtained by taking ``linear parts''.  

\begin{claim}
The orthogonal actions $D\ov\rho$, $D\ov\rho'$ are irreducible.
\end{claim}

\begin{proof}
Since $X$ is non-Haken, the quotient orbifold $\R^2/\ov\rho(\Ga)$ contains no 2-sided, embedded, closed geodesics.
By the classificaton of flat orbifolds, $\R^2/\ov\rho(\Ga)$ must be a 2-sphere with 3 cone points, which implies irreducibility of $D\ov\rho$.

Alternatively, we can argue as follows:
Suppose $D\ov\rho$ is reducible, so each element of $\Gamma$ acts as a reflection, a rotation by $\pi$ or trivially.
If there is an element $g \in \Gamma$ that acts by a reflection, then $\ov\rho(g)$ must be a glide reflection; let $\gamma \subset \R^2$ be its axis.
If there is no such element, then let $\gamma \subset \R^2$ be a line that projects to a circle under the quotient map $\R^2 \to \R^2 / \ov\rho(\Gamma)$.
In both cases, a generic line $\gamma' \subset \R^2$ parallel to $\gamma$ projects to a 2-sided, embedded, geodesic in $\R^2/\ov\rho(\Ga)$, which is a contradiction.

The proof for $D\ov\rho'$ is the same.
\end{proof}

It follows from the Claim that the actions $\Ga\stackrel{D\rho,\,D\rho'}{\acts}L(\nil)$ restricted to the $2$-dimensional subspace $X_3^\perp$, are irreducible.
By \cite{auslander_bieberbach} --- the statement and proof of which are variations on the classical Bieberbach theorem on crystallographic groups ---  there is an affine mapping $\Phi:\nil\ra\nil$ conjugating $\rho$ to $\rho'$.  Since $D\Phi:L(\nil)\ra L(\nil)$ is equivariant with respect to the linear actions $D\rho$, $D\rho'$, it preserves $X_3^\perp$, and induces a conformal mapping $X_3^\perp\ra X_3^\perp$; hence $D\Phi$ has a matrix as in \eqref{eqn_gl2_aut} where $A=\la B$ for some $\la>0$ and $B\in O(2)$, and so $\det D\Phi=\la^4$.
Now  $\vol(\nil/\rho(\Ga))=\vol(\nil/\rho'(\Ga))$ implies that  $|\det D\Phi|= 1$, so $\la=1$ and  $\Phi$ is an isometry.   
\end{proof}

\subsection{Spaces of maps and metrics}
If $X$ is a smooth manifold, we let $\met(X)$ denote the set of smooth Riemannian metrics on $X$ equipped with the $C^\infty$-topology.  
Now assume that $X$ is compact, connected and modeled on $\nil$.  Let $\met_{\nil}(X)$ denote the subspace of $\met(X)$ consisting of unit volume Riemannian metrics isometrically covered by $\nil$ (this is nonempty as mentioned above).
We will need the following result:
\begin{lemma}
\label{lem_structure_met_nil}
Suppose $X$ is non-Haken, and pick $g_X\in \met_{\nil}(X)$.  Then:
\begin{itemize}
\item There is a fibration $\Diff(X)\ra \met_{\nil}(X)$ with fiber homeomorphic to $\isom(X, \linebreak[1] g_X)$.
\item $\Diff(X)$ and $\met_{\nil}(X)$ are homotopy equivalent to CW complexes.
\item $\met_{\nil}(X)$ is contractible if and only if the inclusion $\isom(X,g_X)\ra \Diff(X)$ is a homotopy equivalence.
\end{itemize}
\end{lemma}
This follows from Lemma~\ref{lem_unique_unit_volume_nil} as in the proof of \cite[Lemma 2.2]{bamler_kleiner_gsc}.

\subsection{Singular Ricci flows} 
We will use the terminology associated to the concept of singular Ricci flow, as defined in \cite{bamler_kleiner_gsc}; see also \cite{kleiner_lott_singular_ricci_flows, bamler_kleiner_uniqueness_stability}.
For convenience, we call a Ricci flow spacetime $(\M,\t,\D_\t,g)$, which we often abbreviate by $\M$, a {\bf singular Ricci flow} if it is an $(\eps_{\can},r)$-singular Ricci flow according to \cite[Definition~2.13]{bamler_kleiner_gsc}, for $\eps_{\can} > 0$ as in \cite[Theorem~2.14]{bamler_kleiner_gsc} and for some function $r : [0,\infty) \to (0,\infty)$.
Recall that for every closed, 3-dimensional Riemannian manifold $(M,g)$ there is a singular Ricci flow $\M$ whose initial time-slice $(\M_0,g_0)$ is isometric to $(M,g)$ and this flow is unique up to isometry (existence was established in \cite{kleiner_lott_singular_ricci_flows} and uniqueness in \cite{bamler_kleiner_uniqueness_stability}).
In contrast to Perelman's Ricci flow with surgery \cite{perelman_surgery}, surgeries in $\M$ are not performed at discrete time-steps; instead,  change of topology occurs at an infinitesimal scale.

In the following, we collect the necessary topological and geometric properties of singular Ricci flows that will be needed in this paper. 

If $M$ is a manifold, then a {\bf punctured copy of $M$} is a manifold diffeomorphic to $M\setminus S$, where $S\subset M$ is a finite (possibly empty) subset.  Note that if $M_1$, $M_2$ are compact $3$-manifolds, then punctured copies of $M_1$ and $M_2$ can be diffeomorphic only if $M_1$ is diffeomorphic to $M_2$.  This follows from the fact that if $D$, $D'$ are $3$-disks where $D'\subset \Int D$, then $\ol{D\setminus D'}$ is diffeomorphic to $S^2\times [0,1]$.   Hence the notion of ``filling in'' punctures is well-defined.  

A compact Riemannian manifold  $(M,g)$ is called {\bf $\eps$-almost-flat} if
\begin{equation*} \label{eq_collapsingcondition}
  \sup_{M} |{\Rm}| \big(\diam (M, g) \big)^2 \leq \eps. 
\end{equation*}

\begin{theorem}[Structure of singular Ricci flows]
\label{thm_structure_singular_ricci_flow}
Let $(X,g_0)$ be a compact Riemannian 3-manifold, and let $\M$ be a singular Ricci flow with initial time slice $\M_0=(X,g_0)$.  Then:
\ben
\item For every $t\in [0,\infty)$, each component $C\subset\M_t$ is a punctured copy of some compact $3$-manifold.  
\item Let $\M_t^{\fill}$ be the (possibly empty) $3$-manifold obtained from $\M_t$ by filling in the punctures and throwing away the copies of $S^3$.  Then $\M_t^{\fill}$ is a compact $3$-manifold, i.e., all but finitely many components of $\M_t$ are punctured copies of $S^3$.  Furthermore, for every $t_1<t_2$ the  prime decomposition of $\M_{t_2}^{\fill}$ is part of the prime decomposition of $\M_{t_1}^{\fill}$.  Hence there are only finitely many times at which the  prime decomposition of $\M_t^{\fill}$ changes.
\item $\M_t^{\fill}$ is irreducible and aspherical for large $t$, depending only on the following bounds on the geometry of $\M_0$: 
 upper bounds on the curvature and volume, and a lower bound on the injectivity radius.
  \item 
If $X$ is modeled on $\nil$, then for every $\eps > 0$ there is a time $T(g_0,\eps) <\infty$ such that $\M$ restricted to $[T,\infty)$ is non-singular and $(\M_t, g_t)$ is $\eps$-almost-flat for all $t \geq T$.
Here $T(g_0,\eps)$ may be chosen to depend continuously on the initial metric $g_0$ in the $C^{2}$-topology. 
\een

\end{theorem}
\begin{proof}
Assertions~(1)--(3) are a restatement of \cite[Theorem~2.16]{bamler_kleiner_gsc}.

To see Assertion~(4), note first that the corresponding statement (excluding the last sentence) holds for Ricci flow with surgery due to \cite[Theorem~1.4]{Bamler_longtime_0} and Lemma~\ref{lem_exclude_large_diam} below, assuming that surgeries are performed sufficiently precisely.
We will argue that the proof of \cite[Theorem~1.4]{Bamler_longtime_0} also works for singular Ricci flows.
This is true in general, but we focus on the case in which $(\M_0,g_0)$ is a non-Haken Nil-manifold, because some of the discussion simplifies in this case.

Note first that for large times we still obtain a thick-thin decomposition $\M_t = \M^{\text{thick}}_t \, \dot\cup \, \M^{\text{thin}}_t$, as described in \cite[7.4]{perelman_surgery} or \cite[Proposition 90.1]{kleiner_lott_perelman_notes}, \cite[Proposition~3.16]{Bamler_longtime_A}, where $\M^{\text{thick}}_t$ is diffeomorphic to a hyperbolic manifold whose cuspidal tori are incompressible within $\M_t$.
This can be seen either directly by reproving the necessary estimates in the setting of singular Ricci flows or by approximating $\M$ with Ricci flows with surgery whose surgery parameter goes to zero.
For the latter approach see the end of the proof of \cite[Theorem~2.16]{bamler_kleiner_gsc}.
Since, by assumption, $\M_0$, and thus also $\M^{\fill}_t$, does not have any hyperbolic pieces in its geometric decomposition, we must have $\M^{\text{thick}}_t = \emptyset$ for large $t$.

Next, we can decompose $\M_t = \M^{\text{thin}}_t$ for large $t$, as explained in \cite[Proposition~2.1]{Bamler_longtime_D}.
Note that this result is purely Riemannian, hence it can be directly applied in the setting of singular Ricci flows  at non-singular times, which are dense in $[0, \infty)$ \cite{kleiner_lott_singuylar_ricci_flows_ii}.
The subsequent study of the local collapse and the topological discussion in \cite[Section~2]{Bamler_longtime_D} also applies to non-singular time-slices of $\M$.

Let us now follow the lines of the proof of \cite[Theorem~1.1]{Bamler_longtime_0} in \cite[line~1054]{Bamler_longtime_D}, Case~2 (which is significantly simpler than Case~1).
Apart from what we have already discussed, this proof relies on the following additional ingredients:
\begin{itemize}
\item The curvature bound from \cite[Proposition~4.4]{Bamler_longtime_A}.
The proof of this bound can be translated easily to the setting of singular Ricci flows.
Alternatively, it can again be obtained by approximating $\M$ with Ricci flows with surgery and passing to the limit.
\item The existence of a family of piecewise smooth maps $f_{k,t} : V \to \M_t$ of controlled area, where the simplicial complex $V$ and the homotopy class of $f_{k,t}$ are constructed in \cite[Section~3.7]{Bamler_longtime_C}. 
Note that \cite{Bamler_longtime_C} is purely topological and makes no reference to Ricci flows at all and \cite{Bamler_longtime_B} only concerns classical Ricci flows, except for \cite[Proposition~5.5]{Bamler_longtime_B}, which easily generalizes to singular Ricci flows.
\end{itemize}
The remainder of the arguments in Case~2 of the proof of \cite[Theorem~1.1]{Bamler_longtime_0} also hold for non-singular time-slices of $\M$, showing that there is some time $T' <\infty$ such that $\M_{[T',\infty)}$ can be described by a conventional Ricci flow without singularities and that we have a curvature bound of the form $|{\Rm}| \leq Ct^{-1}$.

After time-shifting $\M_{[T',\infty)}$, we can finally apply \cite[Theorem~1.4]{Bamler_longtime_0} and Lemma~\ref{lem_exclude_large_diam} below to conclude the proof of the first part of Assertion~(4).

It remains to establish the last statement concerning continuity of $T$.
To see this, we use the stability theorem for singular Ricci flows, \cite[Theorem~1.5]{bamler_kleiner_uniqueness_stability} to conclude that 
if $\eps_1>0$ there is a $\delta (g_0, \eps_1) > 0$, which depends continuously on the initial metric $g_0$ in the $C^2$-topology, such that $(\M'_T, g'_T)$ is $\eps_1$-almost-flat for any singular flow $\M'$ with initial time-slice isometric to $(X, g'_0)$ if $\Vert g'_0 - g_0 \Vert_{C^2} < \delta$.  By Theorem~\ref{thm_nil_stability} below, for $\eps_1=\eps_1(\eps) >0$, we obtain that $(\M'_t, g'_t)$ is $\eps$-almost-flat for all $t \geq T$.
This implies that we may choose $T$ to depend continuously on the initial metric, for example,  using a partition of unity.  
\end{proof}

\begin{lemma} \label{lem_exclude_large_diam}
Suppose that $X$ is a compact, non-Haken manifold modeled on $\nil$ and $\pi : \widehat X \to X$ is a finite covering.
Then for any $K < \infty$ there is an $\eps (\pi, K) > 0$ such that there is no Riemannian metric $g$ on $X$ with the following properties:
\begin{enumerate}[label=(\arabic*)]
\item $|{\Rm}_{g}| \leq K$
\item $\diam (\widehat X, \widehat g := \pi^* g) \geq \eps^{-1}$
\item There is a Riemannian metric $\wh g'$ on $\wh X$ such that $(1-\eps) \wh g \leq \wh g' \leq (1+\eps) \wh g$ and such that the following holds: We can represent $\widehat X$ as the total space of a $T^2$-fibration over $S^1$ such that in a fibered neighborhood of each $T^2$-fiber there is a free, $\widehat g'$-isometric $T^2$-action whose orbits are the $T^2$-fibers of this fibration and those orbits have diameter $\leq \eps$.
\end{enumerate}
\end{lemma}
\begin{proof}
Let $\alpha > 0$ be a small constant whose value we will determine in the course of the proof.
Applying \cite{cheeger_fukaya_gromov_nilpotent_structures} to $g$ yields a $(\rho(\alpha,K),k(\alpha))$-round metric $g''$ on $X$ with $(1-\alpha) g \leq g'' \leq (1+\alpha) g$ together with a nilpotent Killing structure $\mathfrak{N}$, whose orbits have diameter $< \alpha$.
The lift $\widehat g'' := \pi^* g''$ is $(\rho,Ck)$-round, it is compatible with the pullback $\widehat{\mathfrak{N}} := \pi^*\mathfrak{N}$ and has orbits of diameter $< C \alpha$, where $C$ may depend on $\pi : \widehat X \to X$.
Note that we still have $(1-\alpha) \widehat g \leq \widehat g'' \leq (1+\alpha) \widehat g$ and thus $(1-2\alpha) \widehat g' \leq \widehat g'' \leq (1+2\alpha) \widehat g'$ for small enough $\eps$.
So if $\eps$ is chosen small enough, then the orbits $\widehat{\mathcal{O}}_p$ of $\widehat{\mathfrak{N}}$ cannot be 3-dimensional, by Assumption~(2), or 0 or 1-dimensional, by Assumption~(3).
So all orbits must by 2-dimensional and therefore isometric to a flat 2-torus or a Klein bottle.
Moreover, these orbits must lie close enough and be homotopic to the orbits of the local isometric $T^2$-action of $\widehat g'$ for sufficiently small $\alpha, \eps$.
We can now exclude orbits $\widehat{\mathcal{O}}_p$ diffeomorphic to the Klein bottle, since any such orbit has a neighborhood that is diffeomorphic to a $\Z_2$-quotient of $T^2 \times [-1,1]$, equivariant with respect to the corresponding $T^2$-action.
So the orbits $\widehat{\mathcal{O}}_p$ are all 2-tori and form a fibration of $\widehat X$ over a circle.
On the other hand, by construction, we have $\pi (\widehat{\mathcal O}_p) = \mathcal{\mathcal O}_{\pi(p)}$ for any $p \in \widehat X$.
So the orbits $\mathcal{O}_p$ of $(X,g'')$ must again be 2-dimensional and, therefore, generic orbits $\mathcal{O}_p$ must be 2-tori that are finitely covered by incompressible 2-torus orbits $\widehat{\mathcal O}_p$ of $(\widehat X, \widehat g'')$.
This, however, contradicts our assumption that $X$ is non-Haken.
\end{proof}

\bigskip
The following theorem is a uniform version of the main result in \cite{Guzhvina}; this slightly improved statement can be easily extracted from the proof in \cite{Guzhvina}.  We have also included a proof for the reader's convenience.

\begin{theorem}[\cite{Guzhvina}]
\label{thm_nil_stability}
For every $\eps>0$ there exists $\de=\de(\eps,n)>0$ such that Ricci flow starting from any $\de$-almost-flat $n$-manifold is immortal  and $\eps$-almost-flat for all $t\geq 0$.
\end{theorem}

Theorem~\ref{thm_nil_stability} follows immediately by taking $\de:=C^{-1}\min(\eps_0,\eps)$ where $\eps_0$ and $C$ are as in Lemma~\ref{lem_nil_stability} below, and applying that lemma iteratively on time intervals $[T_{j-1},T_j]$ where $T_j=T_{j-1}+AK_{T_{j-1}}^{-1}$; note that intervals $[T_{j-1},T_j]$ grow at least geometrically by Assertion~(4) of the lemma. 

\begin{lemma} 
\label{lem_nil_stability}
There exist $\eps_0(n)>0$, $A(n), C(n) <\infty$ with the following property. Suppose that $(M,(g_t)_{t\in [0,T)})$ is  a  Ricci flow on a compact, $n$-manifold with $\eps$-almost-flat initial condition $(M,g_0)$ for some $\eps \leq \eps_0$; here $T \leq \infty$ is chosen maximal.
Then for $K_t:=\sup_M|{\Rm_{g_t}}|$ we have, assuming that $K_0 > 0$:
\begin{enumerate}[label=(\arabic*)]
\item $T > A K_0^{-1}$.
\item $(M,g_t)$ is $C\eps$-almost-flat for all $t\in [0,AK_0^{-1}]$.
\item $(M,g_t)$ is $\frac{\eps}{2}$-almost-flat for $t=AK_0^{-1}$.
\item $K_t\leq \frac12 K_0$ for $t=AK_0^{-1}$.
\end{enumerate}
\end{lemma}

\begin{proof}
We need the following statement.

\begin{claim}
There exist $A' (n), C'(n) <\infty$ with the following property.
Suppose that $(M^*, (g^*_t)_{t \in [0,T^*)})$ is a maximal Ricci flow through left-invariant metrics on a simply connected, nilpotent Lie group with $ |{\Rm_{g^*_0}}| =  1$. 
Then $T^* = \infty$ and the following is true:
\begin{enumerate}[label=(\arabic*)]
\item $ |{\Rm_{g^*_t}}| \leq  C' $ and $|{\Rm_{g^*_t}}| g^*_t \leq   C'  g^*_0$ for all $t >0$.
\item  $ |{\Rm_{g^*_t}}| \leq \frac18 $ and $|{\Rm_{g^*_t}}| g^*_t \leq  \frac1{32}  g^*_0$ for all $t \geq A'$.
\end{enumerate}
\end{claim}

\begin{proof}
This is essentially the content of \cite[Theorem~2.7]{Guzhvina}, which is based on work of Heber and Lauret \cite{Heber_98, Lauret_01}.
If $n=3$, the case of interest for this paper, then the argument is much simpler.
Up to isomorphism, there are only two 3-dimensional simply connected Lie groups: $\R^2$ and $\nil$.
The case $M' = \R^2$ is trivially true and the case $M' = \nil$ follows from a simple computation, see for example \cite{isenberg_jackson}.
\end{proof}

\bigskip
Fix $A', C'$ according to the Claim and suppose that the lemma was false for $A := 16C'A' + 1$ and $C := 16 C'$.
Choose counterexamples $(M_i, (g_{i,t})_{t \in [1, T_i)})$ for a sequence $\eps_i \to 0$.
By parabolic rescaling, we may assume that $\sup_{M_i} |{\Rm_{g_{i,0}}}|=1$.
By applying the maximum principle to $|{\Rm}|$, we obtain a $\tau \in (0,1]$ such that
\begin{equation} \label{eq_Rm_leq_1ot}
 |{\Rm_{g_{i,t}}}|  < 2, \quad \diam(M, g_{i,t}) < 2 \diam (M, g_{i,0}) \qquad \text{for all} \quad t \in [0, \tau]. 
\end{equation}
Choose $T^* \leq \infty$ maximal with the property that for any $T' < T^*$ there is a subsequence such that we have $T' < T_i$ and  $\sup_{M_i, t \in [0,T']} |{\Rm_{g_{i,t}}}| < C(T')$ for some uniform $C(T') < \infty$.
By a simple distance-distortion estimate, we have $\sup_{t \in [0,T']} \diam (M, g_{i,t}) \linebreak[1] \to 0$.
So, after passing to a diagonal subsequence, the universal covers of the flows $(M_i, (g_{i,t})_{t \in (0,T^*)})$, pointed at arbitrary points, converge to a maximal Ricci flow of the form $(M^*, (g^*_{t})_{t \in (0,T^*)})$, which is invariant under a transitive action of a nilpotent group; see \cite{Guzhvina} for details.
By the Claim we have $T^* = \infty$.
Set $K^\infty_\tau := \sup_{M^*} |{\Rm_{g^*_\tau}}| \leq 2$.

\textit{Case 1: $K^\infty_\tau= 0$ \quad} In this case $(M^*, (g^*_{t})_{t \in [\tau,T^*)})$ is flat and we obtain  a contradiction for large $i$.

\textit{Case 2: $K^\infty_\tau \leq (16 C')^{-1}$ \quad}
Consider the time-shifted flow $(M^*, \linebreak[1](g^*_{t-\tau})_{t \geq 0})$.
Rescaling this flow parabolically by the factor $K^\infty_\tau$ yields a Ricci flow that satisfies the assumptions of the Claim.
Assertion~(1) of the Claim provides geometric bounds, which imply that the original flow $(M^*, \linebreak[1](g^*_{t})_{t \geq 0})$ satisfies
\begin{equation} \label{eq_a1}
 |{\Rm_{g^*_t}}| \leq C' K^\infty_\tau , \qquad
|{\Rm_{g^*_t}}|g^*_t \leq C' K^\infty_\tau g^*_\tau 
\end{equation}
for all $t \geq \tau$.
Let us now argue that Assertions~(1)--(4) of the lemma hold for large $i$, which will give us the desired contradiction.
Assertion~(1) holds since $T^* = \infty$.
Assertion~(2) holds for $t \in [0,\tau]$ due to (\ref{eq_Rm_leq_1ot}), which implies that $(M, g_{i,t})$ is $8\eps$-almost-flat.
To see Assertion~(2) for $t \geq \tau$, we use the second bound in (\ref{eq_a1}) to deduce that for large $i$ the almost-flatness increases by at most a factor of $2C'$ between the times $\tau$ and $t$.
Next, note that by (\ref{eq_Rm_leq_1ot}) we get that $(M, g_{i,\tau})$ is even $4 K_\tau^\infty \eps$-almost-flat for large $i$, so by the same argument as before, we obtain that $(M, g_{i,t})$ is even $8 C' K_\tau^\infty \eps$-almost flat.
Since $8 C' K_\tau^\infty \leq \tfrac12$, this implies Assertion~(3) for large $i$.
Assertion~(4) follows from the first bound in (\ref{eq_a1}) for large $i$ since $C' K^\infty_\tau \leq \tfrac1{16}$.

%
%

\textit{Case 3: $K^\infty_\tau > (16 C')^{-1}$ \quad} 
The argument is similar to that in Case 2.
Assertions~(1) and (2) of the lemma follow via the previous argument, since they did not depend on the bound for $K^\infty_\tau$.
To see the remaining assertions, we apply the Claim as in Case~2, but use Assertion~(2), to obtain that
\begin{equation} \label{eq_a2}
 |{\Rm_{g^*_t}}| \leq \tfrac18 K^\infty_\tau \leq \tfrac14, \qquad |{\Rm_{g^*_t}}| g^*_t \leq \tfrac1{32} K^\infty g^*_\tau   
\end{equation}
for $t \geq \tau + A' (K^\infty_\tau)^{-1}$.
Note that by the choice of $A$, these bounds hold for $t = A$. 
As in the previous case, we find that $(M, g_{i,\tau})$ is $8\eps$-almost flat and the second bound in (\ref{eq_a2}) implies that for large $i$ the almost-flatness decreases by a factor of at most $\frac1{16}$ between the times $\tau$ and $A$.
This implies Assertion~(3).
Assertion~(4) follows directly from the first bound in (\ref{eq_a2}) for large $i$.
\end{proof}

\subsection{Improving almost Nil-metrics}

In the following, let $X$ be a closed $3$-manifold modeled on $\nil$.  We will show that the metrics obtained in Theorem~\ref{thm_structure_singular_ricci_flow}(4) can be canonically improved to $\nil$-structures.  The following lemma may be compared with \cite[Remarks 3.4, 6.2]{bamler_kleiner_gsc}.

\begin{lemma}
\label{lem_rounding}
There is an $\eps_0(X) > 0$ with the following property.
Consider the subset $\mathcal{U}_{\eps_0}  \subset \met (X)$  of $\eps_0$-almost-flat metrics. There is a continuous map $\Psi : \mathcal{U}_{\eps_0} \to \met_{\nil}(X)$ such that if $g$ is a deformation of a metric $g^* \in \met_{\nil}(X)$ by the Ricci flow, then $\Psi(g) = g^*$.
\end{lemma}

\begin{proof}
Fix some $g \in \mathcal{U}_{\eps_0}$.
We provide a construction for the metric $\Psi(g) \in  \met_{\nil}(X)$, which depends continuously on $g$.
By replacing $g$ with $\sup_X |{\Rm_g}| \, g$, we may assume that $\sup_X |{\Rm_g}| = 1$.
Next, we may evolve $g$ via the Ricci flow for some uniform time, using Theorem~\ref{thm_nil_stability}, to obtain a metric $g' \in \mathcal{U}_{\eps_0}$ with 
\begin{equation} \label{eq_nabmRm}
|{\nabla^m \Rm}_{g'}| \leq C_m ,
\end{equation}
for some uniform constants $C_m < \infty$.
It is clear that $g'$ depends continuously on $g$.

For notational simplicity we replace $g$ with $g'$. 

Let $\pi : (\wh X, \wh g) \to (X,g)$ be the finite normal covering corresponding to the translation subgroup of $\pi_1(X)\subset \isom(\nil)$ and denote by $G \curvearrowright \wh X$ the action by the finite deck-transformation group. 
Identify $H^1 (\wh X; \R) \cong \R^2$ with the space $\mathcal{H}^1$ of harmonic 1-forms via de Rham cohomology.
Consider the inner product $(\cdot, \cdot)$ on $\mathcal{H}^1$ defined by
\[ (\alpha_1, \alpha_2) := \frac1{\vol(\wh{X}, \wh g)} \int_{\wh X} \langle \alpha_1, \alpha_2 \rangle \, d\wh g, \qquad \alpha_1, \alpha_2 \in \mathcal{H}^1. \]
Let $(Y^2, g_Y)$ be the flat torus obtained by the quotient $$(\mathcal{H}^1)^* / H_1(\wh X; \Z) \cong H_1 (\wh X;\R)/ H_1(\wh X; \Z),$$ equipped with the affine metric induced by $(\cdot, \cdot)$.
We now construct a map $\phi : \wh X \to Y$ as follows.
Fix some $p \in \wh X$.
For any $q \in \wh X$ consider a 1-chain $\gamma \subset \wh X$ with $\partial\gamma = q - p$, which is unique up to a representative of $H_1(\wh X; \Z)$.
Consider the element $h \in (\mathcal{H}^1)^*$ with the property that $h(\alpha) = \alpha(\gamma)$ for any $\alpha \in \mathcal{H}^1$ and define $\phi(q)$ to be the class of $h$ in $(\mathcal{H}^1)^* / H_1(\wh X; \Z)$.

\begin{claim}
\begin{enumerate}[label=(\arabic*)]
\item $\phi : \wh X \to Y$ is smooth and its differential is given by 
\[ d\phi_q : T_q \wh X \longrightarrow T_{\phi(q)} Y \cong (\mathcal{H}^1)^*, \qquad v \longmapsto (\alpha \longmapsto \alpha(v)) \]
\item $\phi$ induces an isomorphism $H_1(\wh X;\Z)\ra H_1(Y;\Z)$.
\item There is an isometric action $G \curvearrowright Y$ for which $\phi$ is equivariant.
\item If $\eps_0 \leq \ov\eps_0(\wh X)$, then $\phi$ is a submersion whose fibers are circles.
\item If $(\wh X, \la^{-2}\wh g)=\nil/\wh\Ga$ for some $\la > 0$, then $\phi$ lifts to the abelianization fibration.
\end{enumerate}
\end{claim}

\begin{proof}
Assertions~(1), (2) are clear by construction.
For Assertion~(3) note that a different choice of basepoint $p$ leads to a map of the form $\phi' = \psi \circ \phi$ for some isometry $\psi : (Y,g_Y) \to (Y, g_Y)$.

To see the first part of Assertion~(4) we argue by contradiction.
So assume that for a sequence $\eps_{0,i} \to 0$ there are metrics $\wh g_i$ such that none of the induced maps $\phi_i : (\wh X, \wh g_i) \to (Y_i, g_{Y_i})$ are submersions.
By (\ref{eq_nabmRm}) we have $\diam(\wh X, \wh g_i) \to 0$.
By Assertion~(1) there must be non-zero harmonic 1-forms $\alpha_i$ on $(\wh X, \wh g_i)$ that vanish at some point $q_i \in \wh X$.
Without loss of generality, we may assume that $\sup_{\wh X} |\alpha_i| = 1$.
Now consider the pullbacks $\td\alpha_i$, $\td g_i$ of $g_i$ and $\alpha_i$ via the exponential map based at $q_i$.
Due to standard gradient estimates and the curvature bounds, we obtain a uniform bound of the form $|\nabla \td\alpha_i| \leq C$ in a fixed ball around the origin.
So $|\nabla \alpha_i| \leq C$ on $\wh X$ for large $i$.
This, however, contradicts the fact that $\sup_{\wh X} |\alpha_i| = 1$ and $|(\alpha_i)_{q_i}| = 0$ for large $i$.  Thus we may assume that $\phi$ is a submersion; the fact that the fibers of $\phi$ are connected follows from Assertion~(2), for example, via the long exact homotopy sequence.

For Assertion~(5) 
note that any constant 1-form on $\R^2$ lifts to a harmonic 1-form on $(\nil, \lambda^2 g_{\nil})$ via the abelianization map $\nil \to \R^2$. Any such lift descends to a harmonic form on $(\wh X, \wh g)$.
Combining this with Assertion~(1) implies Assertion~(5).
\end{proof}

Consider now the $S^1$-fibers of $\phi$.
For any $q \in \wh X$ let $\ell (q)$ be the length of the $S^1$-fiber through $q$.
These fibers represent a generator of the center in $\pi_1 (\wh X)$.
Fixing such a generator determines a unique orientation on the fibers.
Let now $V$ be the uniquely defined vector field that is tangent to these fibers, satisfies $|V| = \ell$ and is positively oriented.
Then the flow of $V$ induces a principal $S^1$-bundle structure on $\wh X$.
We obtain a principal connection $\theta$ on this bundle by averaging $\wh g (\ell^{-1} V, \cdot)$ via the $S^1$-action.
Denote its curvature 2-form by $\omega \in \Omega^2 (Y)$.
Let $\om'\in\Om^2(Y)$ be the parallel $2$-form cohomologous to $\om$, and let $\th'=\th+\phi^*\ov\th$, where $\ov\th$ is the unique element of $\Om^1(Y)$ that is $L^2$-orthogonal to $\ker (d:\Om^1(Y)\ra \Om^2(Y))$ and satisfies $d\ov\th=\om'-\om$.  Let $\ov U_1,\ov U_2$ be a parallel orthonormal frame  on $Y$, and for $a>0$ we let $\wh g_a$ be the Riemannian metric on $\wh X$ such that $U_1, U_2, [U_1,U_2]$ are orthonormal, where $U_1, U_2$ are horizontal lifts with respect to the connection $\th'$ of $a\ov U_1,a\ov U_2$, respectively.   Let $\ov a>0$ be the unique choice of $a$ such that the volume of $\wh g_{\ov a}$ equals the order of the normal covering $\wh X\ra X$.  Note that $\wh g_{\ov a}$ is invariant under the action of the deck group of the covering $\wh X\ra X$, and therefore it descends to a metric on $X$, which we declare to be $\Psi(g)$.

By construction and the Claim, $\Psi(g)$ depends continuously on $g$.  
To see the last statement, observe that $(\wh X,\wh g^*)$ is isometric to a quotient of $\nil$ by a lattice $\wh\Gamma \subset \isom(\nil)$; in the following we identify $(\wh X, \wh g^*) = \nil / \wh \Gamma$.
The Ricci flow $(g_{\nil, t})_{t \geq 0}$ starting from $g_{\nil}$ preserves the symmetries of $g_{\nil}$ and descends to a flow on $\nil / \wh\Gamma$.
Due to the preservation of symmetries, the left-invariant vector fields $X_1, X_2, X_3$ on $\nil$ from Subsection~\ref{subsec_prelim}  remain orthogonal and satisfy $g_{\nil,t}(X_1,X_1) = g_{\nil,t}(X_2,X_2)$.
So by the discussion in Subsection~\ref{subsec_prelim}, we obtain that $g_{\nil,t} = \lambda^2_t \psi^*_t g_{\nil}$ for some family of Carnot dilations $\psi_t$ and scalars $\lambda_t > 0$.
Combining this with Assertion~(5) of the Claim, we obtain that the map $\phi : \nil / \wh \Gamma \to \R^2 / \wh\Gamma$, constructed with respect to any $g_{\nil,t}$, is the standard Seifert fibration, which lifts to the abelianization map $\nil \to \R^2$. Following the construction succeeding the Claim shows that $\Psi(g) = g^*$.
\end{proof}

\section{Proof of Theorem~\ref{thm_main}}
The proof of Theorem~\ref{thm_main} is the same as the proof of the corresponding theorems for spherical space forms and hyperbolic manifolds given in \cite{bamler_kleiner_gsc}, apart from some minor changes, which we now explain.

Section 3 of \cite{bamler_kleiner_gsc} shows that if $X$ is a spherical space form, then we may assign to every $g\in\met(X)$ a partially defined metric $(W_g,\check g)$ with sectional curvature $\equiv 1$, in a continuous manner.
In our case, the corresponding result is:

\begin{lemma}
Let $X$ be a compact, non-Haken manifold modeled on $\nil$, and pick $g_X\in \met_{\nil}(X)$.  Then for every $g\in \met(X)$ there is a  partially defined metric $(W_g,\check g)$ on $X$ (as defined in \cite[Definition~2.3]{bamler_kleiner_gsc}) such that:
\ben
\item
\label{item_isometric_punctured} $(W_g,\check g)$ is isometric to $(X\setminus S_g, g_X)$ for some finite (possibly empty) subset $S_g\subset X$, where the cardinality of $S_g$ is bounded above by a continuous function of $g$ (in the $C^\infty$ topology).
\item
\label{item_nil_same} If $g\in \met_{\nil}(X)$ then $(W_g,\check g)=(X,g)$.  
\item
\label{item_continuous} The assignment $g\mapsto (W_g,\check g)$ is continuous.
\een
\end{lemma}
\begin{proof}
Pick $g\in \met(X)$, and let $\M$ be a singular Ricci flow with $\M_t=(X,g)$.

For every $t_1,t_2\in[0,\infty)$, let $C_{t_1,t_2}\subset C_{t_1}$ be the set of points in $C_{t_1}$ that survive until time $t_2$, i.e. the points for which the time $(t_2-t_1)$-flow of the time vector field $\D_\t$ is defined.  Then $C_{t_1,t_2}$ is an open subset of $C_{t_1}$, and the time $(t_2-t_1)$-flow of $\D_\t$ defines a smooth map $\Phi_{t_1,t_2}:C_{t_1,t_2}\ra \M_{t_2}$, which is a diffeomorphism onto its image.

Now let $T=T(g,\eps_0(X))$, where $\eps_0(X)$ is as in Lemma~\ref{lem_rounding}, and $T(g,\eps_0(X))$ is as in Theorem~\ref{thm_structure_singular_ricci_flow}(4).    Then by Theorem~\ref{thm_structure_singular_ricci_flow}(4) the time-$T$ slice $\M_T$ satisfies the hypotheses of Lemma~\ref{lem_rounding}; we let $\wh g\in \met_{\nil}(\M_T)$ be the $\nil$ metric  supplied by that lemma.  We then define $(W_g,\check g)=(\Phi_{T,0}(C_{T,0}),(\Phi_{T,0})_*\wh g)$.

Properties (\ref{item_isometric_punctured})--(\ref{item_continuous}) now follow as in the spherical case. 
\end{proof}

Sections~4 and 5 of \cite{bamler_kleiner_gsc} now carry over after replacing $\met_{K\equiv 1}(X)$  by $\met_{\nil}(X)$, and the Thurston geometry $S^3$ by  $\nil$.  
Instead of \cite[Lemma~4.3]{bamler_kleiner_gsc} we use:

\begin{lemma}[Extending Nil metrics over a ball]
\label{lem_extending_metrics_ball}
In the following, we let $S^2$ and $D^3$ denote the unit sphere and unit disk in $\R^3$, respectively, and we let $N_r(S^2)$ denote the metric $r$-neighborhood of $S^2\subset \R^3$.

Suppose $m\geq 0$, $\rho>0$ and: 
\begin{enumerate}[label=(\roman*)]
\item $h_{m+1}:D^{m+1}\ra \met_{\nil}(N_{\rho}(S^2)\cap D^3)$ is a continuous map such that for all $p\in D^{m+1}$, the Riemannian manifold $(N_{\rho}(S^2)\cap D^3,h_{m+1}(p))$ isometrically embeds in $\nil$.  Here $\met_{\nil}(N_{\rho}(S^2)\cap D^3)$ is equipped with the $C^\infty_{\loc}$-topology.
\item $\wh h_m:S^m\ra\met_{\nil}(D^3)$ is a continuous map such that for every $p\in S^m$ we have $\wh h_m(p)=h_{m+1}(p)$ on $N_{\rho}(S^2)\cap D^3$, and $(D^3,h_m(p))$ isometrically embeds in $\nil$.
\een

Then, after shrinking $\rho$ if necessary, there is a continuous map $\wh h_{m+1}:D^{m+1}\ra \met_{\nil}(D^3)$ such that:
\begin{enumerate}[label=(\alph*)]
\item $\wh h_{m+1}(p)=h_{m+1}(p)$ on $N_{\rho}(S^2)\cap D^3$ for all $p\in D^{m+1}$.
\item $\wh h_{m+1}(p)=h_m(p)$ for all $p\in S^m$.
\item If for some $p \in D^{m+1}$ there is a metric $g' \in \met_{\nil}(D^3)$ and some $\rho' > 0$ such that $(N_{\rho'}(S^2) \cap D^3, g')$ is isometric to $(N_{\rho'}(S^2) \cap D^3, h_{m+1}(p))$, then $(D^3, g')$ is isometric to $(D^3, \wh{h}_{m+1}(p))$ and thus the volumes of both metrics are the same.
\een
\end{lemma}

\begin{proof}
The proof is the same as that of \cite[Lemma~4.3]{bamler_kleiner_gsc}, except for the first step.
Pick $x \in S^2, x' \in \nil$ and consider the orthonormal basis $e'_1 := (X_1)_{x'}, e'_2 := (X_2)_{x'}, e'_3 := (X_3)_{x'} \in T_{x'} \nil$, where $X_1, X_2, X_3$ denote the left-invariant vector fields from Subsection~\ref{subsec_prelim}.
Similarly, for any $p \in D^{m+1}$ let $f_1 (p) \in T_x \R^3$ such that there is a local $h_{m+1}(p)$-isometric $\nil$-action near $x$ and left-invariant, positively oriented orthonormal frame $E_1, E_2, E_3$ with $E_3 = [E_1,E_2]$ such that $f_3(p) = (E_3)_p$.
This property characterizes $f_3(p)$ uniquely and thus $f_3 : D^{m+1} \to T_x \R^3$ is continuous.
Since $D^{m+1}$ is contractible, we can now find $f_1, f_2 : D^{m+1} \to T_x \R^3$ such that $f_1(p), f_2(p), f_3(p)$ form a positively oriented $h_{m+1}(p)$-orthonormal basis for each $p \in D^{m+1}$.
It follows that there are unique isometric embeddings $\psi_{m+1}(p) : (N_\rho(S^2) \cap D^3, h_{m+1}(p)) \to \nil$, $\wh\psi_m(p) : (D^3, \wh{h}_m (p) ) \to \nil$ sending $f_1(p), f_2(p), f_3(p)$ to $e'_1, e'_2, e'_3$.
The remainder of the argument is the same as in \cite[Lemma~4.3]{bamler_kleiner_gsc}.

Assertion~(c), which is new, follows easily from the uniqueness of isometric embeddings.
\end{proof}

\section*{Acknowledgments}
We thank the anonymous referee for valuable comments.

\bibliography{diffm}{}
\bibliographystyle{amsalpha}

\end{document}